\documentclass[12pt,reqno]{amsart}
\usepackage[left=80pt,right=80pt]{geometry}
\usepackage[usenames]{color}
\usepackage{amsmath}
\usepackage{amssymb}
\usepackage{amsthm}
\usepackage{enumitem}
\usepackage{float}
\usepackage{graphicx}
\usepackage{dsfont}

\usepackage{tikz}
\usepackage{subcaption}
\usepackage{cases}
\usepackage{hyperref,url}
\hypersetup{
	colorlinks=true,
  linkcolor=black,          
  citecolor=black,         
  filecolor=black,      
  urlcolor=black           
}

\newtheorem{prop}{Proposition}

\newtheorem{theorem}[prop]{Theorem}

\newtheorem{corollary}[prop]{Corollary}
\theoremstyle{definition}
\newtheorem{definition}[prop]{Definition}

\newtheorem{example}[prop]{Example}

\newcommand{\s}{\textnormal{s}}
\newcommand{\cs}{\textnormal{cs}}

\newcommand{\seqnum}[1]{\href{https://oeis.org/#1}{\rm \underline{#1}}}
\allowdisplaybreaks
\newcommand{\mylabel}[2]{#2\def\@currentlabel{#2}\label{#1}}
\makeatletter
\newcommand{\vast}{\bBigg@{4}}
\newcommand{\Vast}{\bBigg@{5}}
\makeatother
\begin{document}
\tikzset{mystyle/.style={matrix of nodes,
        nodes in empty cells,
        row 1/.style={nodes={draw=none}},
        row sep=-\pgflinewidth,
        column sep=-\pgflinewidth,
        nodes={draw,minimum width=1cm,minimum height=1cmanchor=center}}}
\tikzset{mystyleb/.style={matrix of nodes,
        nodes in empty cells,
        row sep=-\pgflinewidth,
        column sep=-\pgflinewidth,
        nodes={draw,minimum width=1cm,minimum height=1cmanchor=center}}}

\title{Further results on staircase (cyclic) words}

\author[SELA FRIED]{Sela Fried$^{\dagger}$}
\thanks{$^{\dagger}$ Department of Computer Science, Israel Academic College,
52275 Ramat Gan, Israel.
\\
\href{mailto:friedsela@gmail.com}{\tt friedsela@gmail.com}}

\begin{abstract}
We find the two-variables generating function for the statistic which counts the number of variations in a word bounded by $1$. Thus, we refine and extend previous results concerning staircase words, which are words in which the variation between all consecutive letters is bounded by $1$. We obtain the analogue results for cyclic words.
\bigskip

\smallskip

\noindent
\textbf{Math.~Subj.~Class.:} 68R05, 05A05, 05A15.
\end{abstract}

\maketitle

\section{Introduction}
Let $n\geq 0$ and $k\geq 2$ be two integers. We denote by $[k]$ the set $\{1,2,\ldots,k\}$. A word over $k$ of length $n$ is merely an element of $[k]^n$. A staircase word is a word  such that $|w_{i+1}-w_i|\leq 1$, for every $1\leq i\leq n-1$. If, in addition, $|w_n-w_1|\leq 1$, then $w$ is said to be a cyclic staircase word. 

The study of (cyclic) staircase words seems to have begun with the work of Knopfmacher et al.\ \cite{K}. Generalizations and refinements soon followed, e.g., \cite{FM, Kit, Ka, MRV, Sh}. 

In order to state our results and see how they complement past ones, let us  define the following two statistics: For a word $w=w_1\cdots w_n\in[k]^n$ we set 
\begin{align}
\s(w)&=\{1\leq i\leq n-1\;:\;|w_{i+1}-w_i|\leq 1\},\nonumber\\
\cs(w)&=\{0\leq i\leq n-1\;:\;|w_{i+1}-w_i|\leq 1\},\nonumber
\end{align} 
where by $w_0$ we mean $w_n$. For example, if $w=423353$ then $\s(w)=2$ and $\cs(w)=3$. Notice that for a word $w\in[k]^n$ being staircase (resp., cyclic staircase) is equivalent to $\s(w)=n-1$ (resp., $\cs(w)=n$).

Let $f_{n,k}(t)$ and $g_{n,k}(t)$ denote the distribution of $\s(\cdot)$ and $\cs(\cdot)$ on $[k]^n$, respectively. In this work we establish the two-variable generating functions of the numbers $f_{n,k}(t)$ and $g_{n,k}(t)$, which we denote by $F_k(x,t)$ and $G_k(x,t)$, respectively, and express them in terms of the Chebyshev polynomials. Doing so, we achieve several goals:
\begin{enumerate}
\item We refine the results obtained by Knopfmacher et al.\ \cite{K}, who established the one-variable generating functions of the number of staircase words and cyclic staircase words. 
\item Kitaev and Remmel \cite{Kit} studied $F(x,t)$ but did not obtain a closed-form formula. 
They expressed the generating function as a sum over the entries of the inverse of a certain matrix. Furthermore, they did not study at all cyclic words. We evaluate this expression obtain the analogue statement for cyclic words.
\item In order to obtain an explicit expression of $G(x,t)$, we evaluated in a closed form dozens of sums of products of Chebyshev polynomials. These seem to be of an independent interest.
\end{enumerate}

\section{Preliminaries}

The Chebyshev polynomials are two sequences of polynomials related to the cosine and sine functions, which find extensive use in approximation theory. They also emerge naturally in combinatorics. The Chebyshev polynomials of the first kind, denoted by
$T_n(x)$, are defined by 
$T_{n}(\cos \theta )=\cos(n\theta)$, and satisfy the recursion 
\begin{align}
    T_0(x)&=1\nonumber\\
    T_1(x)&=x\nonumber\\
    T_{n+1}(x)&=2xT_n(x)-T_{n-1}(x).\nonumber
\end{align}
The Chebyshev polynomials of the second kind, denoted by $U_n(x)$, are defined by 
$U_{n}(\cos\theta )\sin\theta=\sin((n+1)\theta)$, and satisfy the recursion 
\begin{align}
    U_0(x)&=1\nonumber\\
    U_1(x)&=2x\nonumber\\
    U_{n+1}(x)&=2xU_n(x)-U_{n-1}(x).\label{u1}
\end{align}
We make extensive use of well-known Chebyshev polynomials identities, for example
$T_{n}(x)=(U_{n}(x)-U_{n-2}(x))/2$  and $U_{nm-1}(x)=U_{m-1}(T_{n}(x))U_{n-1}(x)$ (\cite[Exercise 1.2.15 (a) and (e)]{R}). A particularly useful identity is stated in \cite[(2.4)]{K}:
\[U_{i}(x)U_{j}(x)=\frac{U_{i-j}(x)-xU_{i-j-1}(x)-U_{i+j+2}(x)+xU_{i+j+1}(x)}{2(1-x^{2})}.\] In order to simplify the generating functions we shall need formulas for sums of products of Chebyshev polynomials. Except for $q_0(k,x)$ and $q_2(k,x)$ (\cite[(2.2)]{K} and \cite[p.~45]{CF}, respectively), we could not find any of them in the literature. Let us set \[Z_k(x) = \frac{U_k(x)-U_{k-1}(x)-1}{2(x-1)}.\] We have
\begin{align}
q_0(k,x) & = \sum_{i=1}^{k}U_{i}(x)=\frac{U_{k+1}(x)-U_{k}(x)-1}{2(x-1)}-1\nonumber\\
q_1(k,x) & =\sum_{i=1}^{k}U_{i}^{2}(x)=\frac{U_{2}(x)(U_{2k}(x)-1)-2xU_{2k-1}(x)-2k}{4(x^{2}-1)}\nonumber\\
q_2(k,x) & =\sum_{i=1}^{k}U_{i}(x)U_{k-i}(x)=\frac{kT_{2}(x)U_{k}(x)-(k+1)xU_{k-1}(x)}{2(x^{2}-1)}\nonumber\\
q_3(k,x) & =\sum_{i=1}^{k}U_{i-1}(x)U_{i}(x)=\frac{2xU_{2k}(x)-U_{2k-1}(x)-2(k+1)x}{4(x^{2}-1)}\nonumber\\
q_4(k,x) & =\sum_{i=1}^{k}U_{i-1}(x)U_{k-i}(x)=\frac{kxU_{k}(x)-(k+1)U_{k-1}(x)}{2(x^{2}-1)}    \nonumber\\
q_5(k,x)&=\sum_{i=1}^{k}U_{k-1-i}(x)U_{i-1}(x)=q_{4}(k-1,x)\nonumber\\ 
q_6(k,x)&=\sum_{i=1}^{k}U_{k-1-i}(x)U_{k-i}(x)=q_{3}(k-1,x)\nonumber\\
q_7(k,x)&=\sum_{i=1}^{k}U_{k-i}(x)U_{i}(x)U_{i-1}(x)&\nonumber\\&=\frac{2T_{k}(x)U_{2}(x)Z_{k}(T_{3}(x))+2(1-2T_{k+2}(x))Z_{k-1}(x)}{8(1-x^{2})}&\nonumber\\ &-\frac{4T_{k+2}(x)U_{k-1}(x)+U_{k-1}(x)(Z_{k-1}(T_{3}(x))-Z_{k+1}(T_{3}(x))+1)}{8(1-x^{2})}&\nonumber\\
q_8(k,x)&=\sum_{i=1}^{k}U_{i}(x)U_{k-i-2}(x)U_{i-1}(x)=q_{7}(k-2,x)-\frac{xU_{2k}(x)-U_{2k-1}(x)-x}{2(x^{2}-1)}\nonumber\\
q_9(k,x)&=\sum_{i=1}^{k}U_{k-i}(x)U_{k-i-1}(x)U_{i-1}(x)=q_{7}(k-1,x)\nonumber\\
q_{10}(k,x)&=\sum_{i=1}^{k}U_{k-1-i}(x)U_{k-i}(x)U_{i}(x)=q_{7}(k,x)-\frac{xU_{2k}(x)-U_{2k-1}(x)-x}{2(x^{2}-1)}\nonumber\\
q_{11}(k,x)&=\sum_{i=1}^{k}U_{k-1-i}(x)U_{i-1}(x)U_{i}(x)=q_{7}(k-1,x)\nonumber\\
q_{12}(k,x)&=\sum_{i=1}^{k}U_{k-i}^{2}(x)U_{i}(x)\nonumber\\&=\frac{(2+T_{2k+4}(x))Z_{k+2}(x)-T_{2k+4}(x)U_{2}(x)Z_{k+2}(T_{3}(x))}{4(1-x^{2})}-U_{k}^{2}(x)\nonumber\\&+\frac{U_{2k+3}(x)(Z_{k+3}(T_{3}(x))-Z_{k+1}(T_{3}(x))+Z_{k+1}(x)-Z_{k+3}(x))}{8(1-x^{2})}\nonumber\\
q_{13}(k,x)&=\sum_{i=1}^{k}U_{i}^{2}(x)U_{k-i}(x)=q_{12}(k,x)+\frac{T_{2}(x)U_{2k}(x)-xU_{2k-1}(x)-2(x^{2}-1)U_{k}(x)-1}{2(x^{2}-1)}\nonumber\\
q_{14}(k,x)&=\sum_{i=1}^{k}U_{i-1}^{2}(x)U_{k-i-1}(x)=q_{12}(k-2,x)+\frac{T_{2}(x)U_{2k}(x)-T_{3}(x)U_{2k-1}(x)-1}{2(x^{2}-1)}\nonumber\\
q_{15}(k,x)&=\sum_{i=1}^{k}U_{i-1}(x)U_{k-i-1}^{2}(x)=q_{14}(k,x)\nonumber\\
q_{16}(k,x)&=\sum_{i=1}^{k}U_{i}^{2}(x)U_{k-i}^{2}(x)\nonumber\\&=\frac{x((16kx^{4}-8(2k+1)x^{2}+2k)U_{2k}(x)+8x^{2}+4k)}{16x(x^{2}-1)^{2}}-\frac{(8kx^{4}-6(k+1)x^{2}+1)U_{2k-1}(x)}{16x(x^{2}-1)^{2}}\nonumber\\
q_{17}(k,x)&=\sum_{i=1}^{k}U_{i-1}^{2}(x)U_{k-1-i}^{2}(x)=q_{16}(k-2,x)+\frac{T_{2}(x)U_{2k}(x)-T_{3}(x)U_{2k-1}(x)-1}{2(x^{2}-1)}\nonumber\\
q_{18}(k,x)&=\sum_{i=1}^{k}U_{i}(x)U_{i-1}(x)U_{k-i}(x)U_{k-i-1}(x)\nonumber\\&=\frac{\left(4(k-1)x^{2}-2k+3\right)U_{2k}(x)-(2k-1)xU_{2k-1}(x)+4(k+1)x^{2}}{16(x^{2}-1)^{2}}\nonumber\\
&-\frac{Z_{k+1}(T_{2}(x))-Z_{k-1}(T_{2}(x))}{4(x^{2}-1)^{2}}+\frac{T_{2k}(x)(Z_{k+1}(T_{4}(x))-Z_{k-1}(T_{4}(x)))}{16(x^{2}-1)^{2}}\nonumber\\
&+\frac{(1-T_{2}^{2}(x))T_{2}(x)U_{k-1}(T_{2}(x))Z_{k}(T_{4}(x))}{4(x^{2}-1)^{2}}\nonumber\\
q_{19}(k,x)&=\sum_{i=1}^{k}U_{i}(x)U_{k-i-2}(x)U_{i-1}(x)U_{k-i-1}(x)=q_{18}(k-1,x)\nonumber\\
q_{20}(k,x)&=\sum_{i=1}^{k}U_{i}(x)U_{k-i-2}(x)U_{k-i}(x)\nonumber\\&=U_{2}(x)q_{12}(k-2,x)-2xq_{7}(k-2,x)\nonumber\\&+\frac{xT_{3}(x)U_{2k}(x)-xT_{4}(x)U_{2k-1}(x)-(x^{2}-1)U_{k}(x)-x^{2}}{x^{2}-1}\nonumber\\
q_{21}(k,x)&=\sum_{i=1}^{k}U_{i}^{2}(x)U_{k-i-2}(x)U_{k-i}(x)\nonumber\\&=\frac{(16x^{4}-2\left(k+8\right)x^{2}+1)U_{2k-1}(x)}{16x(x^{2}-1)^{2}}\nonumber\\
&+\frac{x(2(-8x^{4}+2(k+3)x^{2}-k)U_{2k}(x)+2(2k+1)T_{2}(x)+2T_{4}(x))}{16x(x^{2}-1)^{2}}\nonumber\\
q_{22}(k,x)&=\sum_{i=1}^{k}U_{k-1-i}^{2}(x)U_{i+1}(x)U_{i-1}(x)=q_{21}(k,x)+\frac{2x^{2}U_{2k}(x)-2xU_{2k-1}(x)+T_{2}(x)-U_{2}(x)}{2(x^{2}-1)}\nonumber\\
q_{23}(k,x)&=\sum_{i=1}^{k}U_{k-1-i}(x)U_{i+1}(x)U_{k-i}(x)\nonumber\\&=q_{7}(k+1,x)+\frac{(T_{2}(x)+2x)U_{2k-1}(x)-(T_{3}(x)+2x^{2})U_{2k}(x)+2x^{2}+x}{2(x^{2}-1)}\nonumber\\
q_{24}(k,x)&=\sum_{i=1}^{k}U_{k-1-i}(x)U_{i+1}(x)U_{i-1}(x)\nonumber\\&=q_{12}(k-2,x)-2xq_{7}(k-2,x)\nonumber\\&+\frac{(-16x^{5}+16x^{3}-3x)U_{2k-1}(x)+(8x^{4}-6x^{2}+1)U_{2k}(x)-U_{2}(x)}{2(x^{2}-1)}\nonumber\\
q_{25}(k,x)&=\sum_{ii=1}^{k}U_{i-1-m}(x)U_{i+1}(x)U_{k-i}(x)U_{m}(x)=q_{18}(k+1,x)-\frac{x(xU_{2k}(x)-U_{2k-1}(x)-x)}{x^{2}-1}\nonumber\\
q_{26}(k,x)&=\sum_{i=1}^{k}U_{i-1}^{2}(x)=q_{1}(k-1,x)+1\nonumber\\
q_{27}(k,x)&=\sum_{i=1}^{k}U_{k-i}^{2}(x)=q_{1}(k-1,x)+1\nonumber\\
q_{28}(k,x)&=\sum_{i=1}^{k}U_{i-1}(x)U_{k-i-2}(x)=q_{4}(k-2,x)-U_{k-1}(x)\nonumber\\
q_{29}(k,x)&=\sum_{i=1}^{k}U_{k-3-i}(x)U_{i+1}(x)U_{i-1}(x)\nonumber\\&=U_{2}(x)q_{12}(k-4,x)-2xq_{7}(k-4)\nonumber\\&-\frac{x\left(128x^{8}-64x^{6}+168x^{4}-42x^{2}-x+3\right)U_{2k-1}(x)}{x^{2}-1}\nonumber\\&+\frac{x\left(\left(-64x^{7}+112x^{5}-60x^{3}+2x^{2}+11x-1\right)U_{2k}(x)+2x^{2}+x-U_{2}(x)\right)}{x^{2}-1}\nonumber\\
q_{30}(k,x)&=\sum_{i=1}^{k}U_{i+1}(x)U_{i}(x)U_{k-i}(x)=q_{7}(k+1,x)-2xU_{k}(x)\nonumber\\
q_{31}(k,x)&=\sum_{i=1}^{k}U_{k-i}(x)U_{i+1}(x)=q_{2}(k+1,x)-2xU_{k}(x)\nonumber\\
q_{32}(k,x)&=\sum_{i=1}^{k}U_{k-i}(x)U_{k-i-2}(x)U_{i+2}(x)\nonumber\\&=q_{20}(k+2,x)\nonumber\\&+\frac{(2x^{2}+xU_{2}(x))U_{2k-1}(x)+(-4x^{3}+2x-U_{2}(x))U_{2k}(x)-(2x+U_{2}(x))(2x^{2}-U_{2}(x))}{2(x^{2}-1)}\nonumber\\
q_{33}(k,x)&=\sum_{i=1}^{k}U_{i-1}(x)U_{k-i-2}(x)U_{k-i-1}(x)=q_{7}(k-2,x)\nonumber\\
q_{34}(k,x)&=\sum_{i=1}^{k}U_{i-1}(x)U_{i+1}(x)=\frac{U_{2}(x)U_{2k}(x)-2xU_{2k-1}(x)-2\left(k+1\right)T_{2}(x)-1}{4(x^{2}-1)}\nonumber\\
q_{35}(k,x)&=\sum_{i=1}^{k}U_{k-i}(x)U_{k-i-2}(x)=q_{34}(k-2,x)-1\nonumber\\
q_{36}(k,x)&=\sum_{i=1}^{k}U_{k-i}(x)U_{i}(x)U_{k-i-2}(x)U_{i+2}(x)\nonumber\\&\hspace{-30pt}=\frac{(64x^{6}-8(k+11)x^{4}+6(k+5)x^{2}-1)U_{2k-1}(x)}{16x(x^{2}-1)^{2}}\nonumber\\&\hspace{-30pt}+	
\frac{x\left(\left(-72x^{6}+(16k+90)x^{4}-\left(16k+24\right)x^{2}+2k\right)U_{2k}(x)+2(k+1)T\left(4,x\right)+4x^{2}+2k+2T_{6}(x)\right)}{16x(x^{2}-1)^{2}}	\nonumber\\
q_{37}(k,x)&=\sum_{i=1}^{k}U_{i-1}i(x)U_{k-i-1}(x)U_{k-i-3}(x)U_{i+1}(x)\nonumber\\&=q_{36}(k-2,x)-\frac{U_{2}(x)(T_{5}(x)U_{2k-1}(x)-T_{4}(x)U_{2k}(x)+U_{2}(x)-2x^{2})}{2(x^{2}-1)}\nonumber\\
q_{38}(k,x)&=\sum_{i=1}^{k}U_{i-1}(x)U_{k-i-2}(x)U_{i+2}(x)\nonumber\\&=2xq_{24}(k-1,x)-q_{23}(k-3,x)+\frac{(64x^{7}-96x^{5}+8x^{4}+36x^{3}-6x^{2}-2x)U_{2k-1}(x)}{2(x^{2}-1)}\nonumber\\&+	
\frac{(-32x^{6}+40x^{4}-8x^{3}-10x^{2}+6x)U_{2k}(x)+2x^{2}-xU_{2}(x)+U_{3}(x)+x}{2(x^{2}-1)}	\nonumber\\
q_{39}(k,x)&=\sum_{i=1}^{k}U_{i+1}(x)U_{k-i}(x)U_{k-i-3}(x)\nonumber\\&=2xU_{2}(x)q_{12}(k-2,x)-4x^{2}q_{7}(k-2,x)-q_{7}(k-1,x)\nonumber\\&+\frac{(16x^{5}-12x^{3}-2x^{2}+x)U_{2k}(x)+(-32x^{6}+32x^{4}+4x^{3}-6x^{2}-2x+1)U_{2k-1}(x)}{2(x^{2}-1)}	
\nonumber\\&+\frac{(-8x^{4}-4x^{3}+8x^{2}+4x)U_{k}(x)+(4x^{3}-4x)U_{k-1}(x)+(2x+1)U_{3}(x)-(2x^{2}+3x)U_{2}(x)}{2(x^{2}-1)}	\nonumber\\
q_{40}(k,x)&=\sum_{i=1}^{k}U_{k-i}(x)U_{i}(x)U_{i+1}(x)U_{k-i-3}(x)\nonumber\\&=\frac{(-1+64x^{6}-88x^{4}-2(k-13)x^{2})U_{2k-1}(x)}{16x^{5}-32x^{3}+16x}	
\nonumber\\&+\frac{x((-64x^{6}+80x^{4}+\left(4k-20\right)x^{2}-2k)U_{2k}(x)+2(k+1)T_{4}(x)+2kT_{2}(x)+2T_{6}(x))}{16x^{5}-32x^{3}+16x}	\nonumber\\
q_{41}(k,x)&=\sum_{i=1}^{k}U_{i-1}(x)U_{k-i-1}(x)U_{k-i-2}(x)U_{i+2}(x)\nonumber\\&=q_{40}(k,x)+\frac{x((4x^{3}-x)U_{2k}(x)-U_{2}(x)U_{2k-1}(x)+xU_{2}(x)-U_{3}(x)-2x)}{x^{2}-1}\nonumber
\end{align}

Finally, let us introduce some notation. If $v$ is a vector, we denote by $v^T$ the transpose of $v$. We denote by $\mathbf{1}=(1,\ldots,1)^T$ the all $1$s vector of length $k$. If $p$ is a condition, the indicator function $\mathds{1}_p$ is equal to $1$ if $p$  holds true and $0$ otherwise. If $M$ is a matrix, we write $(M)_{ij}$ for the $ij$th entry of $M$.

\section{Main results}

In the following theorem we calculate $F(x,t)$. While it was already calculated in \cite[Theorem 4]{Kit}, we obtain an elegant closed form formula for $F(x,t)$ in terms of the Chebyshev polynomials. 

\begin{theorem}\label{t1}
Let $\phi = (1-x(t-1))/(2x(t-1))$. Then
$F(x,t) = 1/(1-x\gamma(x,t))$, where \[\gamma(x,t) = \frac{k}{1-3x(t-1)}-\frac{2x(t-1)}{(1-3x(t-1))^{2}}\frac{U_{k}(\phi)-U_{k-1}(\phi)-1}{U_{k}(\phi)}.\]
\end{theorem}

\begin{proof}

For $i\in[k]$ let $f_i(x,t)$ be the analogue of $F(x,t)$ for words whose first letter is $i$. We set $f(x,t)=(f_{1}(x,t),\ldots,f_{k}(x,t))^T$. We have
\begin{equation}\label{7jv}
f_i(x,t)=x+x\sum_{j=1}^{k}t^{\mathds{1}_{|i-j|\leq1 }}f_j(x,t).
\end{equation}
In matrix form, \eqref{7jv} may be written as $A(x,t)f(x,t)=x\mathbf{1}$, where $A(x,t)$ is the square matrix of size $k$ given by \begin{align}
A(x,t)&=
\begin{pmatrix}
1-tx& -tx& -x& \cdots& -x\\  
-tx& 1-tx& -tx& \ddots& \vdots \\  
-x& \ddots& \ddots& \ddots& -x  \\
\vdots& \ddots& -tx& 1-tx& -tx  \\
-x& \cdots& -x& -tx& 1-tx
\end{pmatrix}.\nonumber
\end{align} 
Since $F(x,t) = \sum_{i=1}^kf_i(x,t) = \mathbf{1}^Tf(x,t)$, we have $F(x,t)=1+x\mathbf{1}^TA(x,t)^{-1}\mathbf{1}$ (this is where the calculation in \cite[Theorem 4]{Kit} terminates). Now, we notice that $A(x,t)$ may be written as
$A(x,t)=C(x(t-1))+(-x)\mathbf{1}\mathbf{1}^T$, where $C(x)$ is the square matrix of size $k$ given by
\begin{align}
C(x)&=\begin{pmatrix}
1-x& -x& 0& \cdots& 0\\  
-x& 1-x& -x& \ddots& \vdots \\  
0& \ddots& \ddots& \ddots& 0  \\
\vdots& \ddots& -x& 1-x& -x  \\
0& \cdots& 0& -x& 1-x
\end{pmatrix}.\nonumber
\end{align}
The matrix $C(x)$ was analyzed by Knopfmacher et al.\ \cite{K}, who showed that its inverse is given by
\[\left(C(x)^{-1}\right)_{ij}=\frac{1}{xU_{k}\left(\frac{1-x}{2x}\right)}\begin{cases}
U_{i-1}\left(\frac{1-x}{2x}\right)U_{k-j}\left(\frac{1-x}{2x}\right), & \text{if \ensuremath{i\leq j}};\\
U_{j-1}\left(\frac{1-x}{2x}\right)U_{k-i}\left(\frac{1-x}{2x}\right), & \text{otherwise}.
\end{cases}\]
By the Sherman–Morrison formula \cite{B}, we have
\begin{equation}\label{e3}
A(x,t)^{-1} = C(x(t-1))^{-1} + \frac{xC(x(t-1))^{-1}\mathbf{1}\mathbf{1}^TC(x(t-1))^{-1}}{1-x\mathbf{1}^TC(x(t-1))^{-1}\mathbf{1}}.
\end{equation} We claim that \begin{equation}\label{e4}
\mathbf{1}^{T}C(x(t-1))^{-1}\mathbf{1}=\gamma(x,t).\end{equation} Indeed,
\begin{align}
&\mathbf{1}^{T}C(x(t-1))^{-1}\mathbf{1}\nonumber\\
&=\sum_{i=1}^{k}\sum_{j=1}^{k}(C(x(t-1))^{-1})_{ij}\nonumber\\
&=\frac{1}{x(t-1)U_{k}(\phi)}\left(\sum_{i=1}^{k}U_{i-1}(\phi)\sum_{j=i}^{k}U_{k-j}(\phi)+\sum_{i=2}^{k}U_{k-i}(\phi)\sum_{j=1}^{i-1}U_{j-1}(\phi)\right)\nonumber\\
&=\frac{1}{x(t-1)U_{k}(\phi)}\left(\sum_{i=1}^{k}U_{i-1}(\phi)(q_{0}(k-i,\phi)+1)+\sum_{i=2}^{k}U_{k-i}(\phi)(q_{0}(i-2,\phi)+1)\right)\nonumber\\
&=\frac{1}{2x(t-1)(\phi-1)U_{k}(\phi)}\Bigg(\sum_{i=1}^{k}\left(U_{i-1}(\phi)U_{k-i+1}(\phi)-U_{i-1}(\phi)U_{k-i}(\phi)-U_{i-1}(\phi)\right)\nonumber\\&\hspace{128pt}+\sum_{i=2}^{k}\left(U_{k-i}(\phi)U_{i-1}(\phi)-U_{k-i}(\phi)U_{i-2}(\phi)-U_{k-i}(\phi)\right)\Bigg)\nonumber\\
&=\frac{1}{2x(t-1)(\phi-1)U_{k}(\phi)}\Bigg((q_{4}(k+1,\phi)-q_{4}(k-1,\phi)-q_{0}(k-1,\phi)\nonumber\\&\hspace{130pt}-q_{0}(k-2,\phi)-U_{k}(\phi)-U_{k-1}(\phi)-2\Bigg)\nonumber\\
&=\frac{1}{4x(t-1)(\phi-1)^{2}(\phi+1)U_{k}(\phi)}\Bigg((2(1-\phi^{2})-(k-1)\phi)U_{k-1}(\phi)+(\phi+1+k)U_{k-2}(\phi)\nonumber\\&\hspace{150pt}+(k+1)\phi U_{k+1}(\phi)-(2\phi^{2}+\phi+1+k)U_{k}(\phi)+2(\phi+1)\Bigg).\nonumber
\end{align}
Using \eqref{u1}, we replace $U_{k-2}(\phi)$ with $2\phi U_{k-1}(\phi)-U_{k}(\phi)$ and $U_{k+1}(\phi)$ with $2\phi U_{k}(\phi)-U_{k-1}(\phi)$, and the assertion follows with some algebra.
Set
\[\alpha_{i}(x)=\frac{U_{k-i}(x)(U_i(x)-1)-U_{i-1}(x)\left(U_{k-i-1}(x)+1\right)}{U_k(x)}.\] We claim that 
\begin{equation}\label{e5}\left(C(x(t-1))^{-1}\boldsymbol{1}\boldsymbol{1}^TC(x(t-1))^{-1}\right)_{ij}=\frac{\alpha_i(\phi)\alpha_j(\phi)}{(1-3x(t-1))^{2}}.\end{equation} Indeed,
\begin{align}
(C(x(t-1))^{-1}\boldsymbol{1})_{i}&=\sum_{s=1}^{k}\left(C(x(t-1))^{-1}\right)_{is}\nonumber\\
&=\frac{1}{x(t-1)U_{k}(\phi)}\left(U_{k-i}(\phi)\sum_{s=1}^{i-1}U_{s-1}(\phi)+U_{i-1}(\phi)\sum_{s=i}^{k}U_{k-s}(\phi)\right)\nonumber\\
&=\frac{1}{x(t-1)U_{k}(\phi)}\left(U_{k-i}(\phi)(q_{0}(i-2,\phi)+1)+U_{i-1}(\phi)(q_{0}(k-i,\phi)+1)\right)\nonumber\\
&=\frac{1}{2(\phi-1)x(t-1)U_{k}(\phi)}\left(U_{i-1}(\phi)(U_{k-i+1}(\phi)-1)-U_{k-i}(\phi)(U_{i-2}(\phi)+1)\right)\nonumber
\end{align}
Using \eqref{u1}, we replace $U_{k-i+1}(\phi)$ with $2\phi U_{k-i}(\phi)-U_{k-i-1}(\phi)$ and $U_{i-2}(\phi)$ with $2\phi U_{i-1}(\phi)-U_{i}(\phi)$. This gives us\[(C(x(t-1))^{-1}\boldsymbol{1})_{i}=\frac{\alpha_i(\phi)}{1-3x(t-1)}.\]
Similarly, \[\left(\boldsymbol{1}^TC(x(t-1))^{-1}\right)_j=\frac{\alpha_j(\phi)}{1-3x(t-1)}.\]
Since \[\left(C(x(t-1))^{-1}\boldsymbol{1}\boldsymbol{1}^TC(x(t-1))^{-1}\right)_{ij}=(C(x(t-1))^{-1}\boldsymbol{1})_{i}\left(\boldsymbol{1}^TC(x(t-1))^{-1}\right)_j,\] the assertion follows.
Finally, we claim that 
\[\sum_{i=1}^k \alpha_i(x)=k-\frac{U_{k}(x)-U_{k-1}(x)-1}{(x-1)U_{k}(x)}.\] Indeed, 
\begin{align}
\sum_{i=1}^{k}\alpha_{i}(x)	&=\frac{1}{U_{k}(x)}\sum_{i=1}^{k}(U_{k-i}(x)(U_{i}(x)-1)-U_{i-1}(x)(U_{k-i-1}+1))\nonumber\\
&=\frac{1}{U_{k}(x)}\sum_{i=1}^{k}(q_{2}(k,x)-q_{4}(k-1,x)-2(q_{0}(k-1,x)+1))\nonumber\\
&=\frac{(kT_{2}(x)-2(x+1))U_{k}(x)+2(1-(k-1)x)U_{k-1}(x)+kU_{k-2}(x)+2(x+1)}{2(x^{2}-1)}.\nonumber    
\end{align}
Replacing $U_{k-2}(x)$ with $2xU_{k-1}(x)-U_k(x)$ and $T_2(x)$ with $2x^2-1$, the assertion follows.
Putting everything together, 
\begin{align}
&F(x,t)\nonumber\\
&=1+x\mathbf{1}^TA(x,t)^{-1}\mathbf{1}\nonumber\\
&=1+x\mathbf{1}^TC(x(t-1))^{-1}\mathbf{1} + \frac{x^2\sum_{i,j=1}^k(C(x(t-1))^{-1}\boldsymbol{1}\boldsymbol{1}^{t}C(x(t-1))^{-1})_{ij}}{1-x\mathbf{1}^TC(x(t-1))^{-1}\mathbf{1}}
\nonumber\\
&=1+x\mathbf{1}^TC(x(t-1))^{-1}\mathbf{1} + \frac{x^2}{(1-3x(t-1))^{2}}\frac{\sum_{i,j=1}^k\alpha_i(\phi)\alpha_j(\phi)}{1-x\mathbf{1}^TC(x(t-1))^{-1}\mathbf{1}}
\nonumber\\
&=1+x\gamma(x,t)+\frac{x^{2}(\gamma(x,t))^{2}}{1-x\gamma(x,t)}\nonumber\\
&=\frac{1}{1-x\gamma(x,t)}.\nonumber\qedhere
\end{align} 
\end{proof}

As a special case of Theorem \ref{t1} we obtain \cite[Theorem 2.2]{K}.
\begin{corollary}
Let $D(x)$ be the generating function of the number of staircase words over $k$. Then \[D(x) = 1+\frac{x(k-(3k+2)x)}{(1-3x)^{2}}+\frac{2x^{2}}{(1-3x)^{2}}\frac{U_{k-1}\left(\frac{1-x}{2x}\right)+1}{U_k\left(\frac{1-x}{2x}\right)}.\]    
\end{corollary}

\begin{proof}
Let $f_{m,n}$ denote the number of words $w\in[k]^n$ such that $\s(w)=m$. Thus, $D(x)= 1+\sum_{n\geq 1}f_{n-1,n}x^n$. Now, we have \[\frac{F(x,t)-1}{x}=\sum_{n\geq 1}\sum_{m=0}^{n-1}f_{m,n}t^mx^{n-1}.\] Thus, \[\frac{F\left(tx,\frac{1}{t}\right)-1}{tx}=\sum_{n\geq1}\sum_{m=0}^{n-1}f_{n,m}t^{n-1-m}x^{n-1}.\] Therefore, 
\begin{align}
D(x)&=1+\frac{F\left(tx,\frac{1}{t}\right)-1}{tx}\Bigg|_{t=0}\nonumber\\&=1+\frac{x(k-(3k+2)x)}{(1-3x)^{2}}+\frac{2x^{2}}{(1-3x)^{2}}\frac{U_{k-1}\left(\frac{1-x}{2x}\right)+1}{U_{k}\left(\frac{1-x}{2x}\right)}.\nonumber \qedhere
\end{align}
\end{proof}

Recall that the Hertzsprung's problem asks for the number of ways to arrange $n$ non-attacking kings on an $n\times n$ board, such that each row and each column contains exactly one king (see \seqnum{A002464}). This problem is clearly equivalent to the problem of finding the number of permutations of $\{1,2,\ldots,n\}$ such that consecutive numbers differ by at least $2$. We propose the name Hertzsprung for the analogue problem for words.

\begin{definition}
A word $w$ over $k$ is said to be Hertzsprung if $\s(w)=0$. We say that $w$ is cyclic Hertzsprung if $\cs(w)=0$.      
\end{definition}

\begin{corollary}
The generating function of the number of Hertzsprung words over $k$ is given by \[\left(1-\frac{kx}{1+3x}-\frac{2x^{2}}{(1+3x)^{2}}\frac{U_{k}\left(-\frac{1+x}{2x}\right)-U_{k-1}\left(-\frac{1+x}{2x}\right)-1}{U_{k}\left(-\frac{1+x}{2x}\right)}\right)^{-1}.\] 
\end{corollary}

\begin{proof}
The assertion follows immediately from evaluating $F(x,0)$. 
\end{proof}

In the following theorem we calculate $G(x,t)$. To the best of our knowledge, it has not been previously studied, except in the special case of cyclic staircase words by \cite[Theorem 3.2]{K}.

\begin{theorem}\label{t2}
We have 
\begin{align}
&G(x,t)=1+\frac{1+3x(t-1)}{(1-3x(t-1))(1+x(t-1))}\vast[\frac{1}{1-x\gamma(x,t)}\nonumber\\&-\frac{2x(t-1)(k+1)}{(1+3x(t-1))U_{k}\left(\frac{1-x(t-1)}{2x(t-1)}\right)}\left(\frac{1-x\frac{1+k-U_{k}\left(\frac{1-x(t-1)}{2x(t-1)}\right)}{1-3x(t-1)}}{1-x\gamma(x,t)}+U_{k-1}\left(\frac{1-x(t-1)}{2x(t-1)}\right)-1\right)+kx(t-1)-1\vast].\nonumber
\end{align}
\end{theorem}

\begin{proof}
For $i,j\in[k]$ let $g_{i,j}(x,t)$ be the analogue of $G(x,t)$ for words whose first and last letters are $i$ and $j$, respectively.   We have
\begin{equation}\label{7ja}
g_{i,j}(x,t)=t^{2\cdot \mathds{1}_{|i-j|\leq 1}}x^{2}+x\sum_{s=1}^{k}t^{\mathds{1}_{|i-j|\leq 1}+\mathds{1}_{|j-s|\leq 1}-\mathds{1}_{|i-s|\leq 1}}g_{i,s}(x,t).
\end{equation}
Dividing both sides of \eqref{7ja} by $t^{\mathds{1}_{|i-j|\leq 1}}$ and setting $h_{i,j}(x,t)=g_{i,j}(x,t)/t^{\mathds{1}_{|i-i|\leq 1}}$, we may rewrite \eqref{7ja} as
\begin{equation}\label{906}
h_{i,j}(x,t)=t^{\mathds{1}_{|i-j|\leq 1 }}x^{2}+x\sum_{s=1}^{k}t^{\mathds{1}_{|j-s|\leq 1}}h_{i,s}(x,t).
\end{equation}
Let us order the pairs $(i,j)\in [k]^2$ lexicographically by $$(1,1) < (1,2) <\cdots < (1,k) < (2,1) < \cdots < (k,k).$$ In matrix form \eqref{906} may be written as $M(x,t)h(x,t) = x^2v(t)$, where $M(x,t)$ is the block diagonal matrix consisting of $k$ copies of $A(x,t)$, $h(x,t)=(h_{1,1}(x,t),\ldots,h_{k,k}(x,t))^T$,
and $v(t)=(v_{1,1}(t),\ldots,v_{k,k}(t))^T$ with $v_{i,j}(t) = t^{\mathds{1}_{|i-j|\leq 1}}$. Finally, for $i\in[k]$ we set $u_i(t) = (v_{i, 1}(t),\ldots, v_{i, k}(t))^T$. Then 
\begin{align}
&G(x,t) \nonumber\\&
= 1+ktx+x^2\sum_{i,j=1}^k g_{i,j}(x,t) \nonumber\\
&=1+ktx+x^2v(t)^Th(x,t)\nonumber\\
&=1+ktx+x^2v(t)^TM(x,t)^{-1}v(t)\nonumber\\
&=1+ktx+x^{2}\sum_{i=1}^{k}u_i(t)^{T}A(x,t)^{-1}u_i(t)\nonumber\\
&=1+ktx+x^{2}\sum_{i=1}^{k}\sum_{j=1}^{k}v_{i,j}(t)\sum_{s=1}^{k}(A(x,t)^{-1})_{js}v_{is}(t)\nonumber\\
&=1+ktx+x^{2}\sum_{i=1}^{k}\left(t^{2}\sum_{\substack{|i-j|\leq1\\|i-s|\leq 1}}(A(x,t)^{-1})_{js}+2t\sum_{\substack{|i-j|\leq1\\|i-s|>1}}(A(x,t)^{-1})_{js}+\sum_{\substack{|i-j|>1\\|i-s|>1}}(A(x,t)^{-1})_{js}\right).\nonumber
\end{align}
Consider now the three multisets (i.e., repetition matters)
\begin{align}
X_1 & = \{(j,s)\;:\; i,j,s\in[k], |i-j|\leq 1 \textnormal{ and }    |i-s|\leq 1\}\nonumber \\
X_2 & = \{(j,s)\;:\; i,j,s\in[k], |i-j|\leq 1 \textnormal{ and }    |i-s|> 1\}\nonumber \\
X_3 & = \{(j,s)\;:\; i,j,s\in[k], |i-j|> 1 \textnormal{ and }    |i-s|> 1\}.\nonumber 
\end{align}
Each of these multisets may be represented by a matrix (that we denote by the same name as the multiset), whose $js$th entry corresponds to the multiplicity of the element $(j,s)$ in the multiset. The matrices are defined as follows (see Example \ref{e1}):
\begin{align}
(X_1)_{js}&=
\begin{cases}
3, &\textnormal{if } 2\leq  s=j\leq k-1  \\ 
2, &\textnormal{if } (j,s)\in\{(1,1),(k,k)\} \textnormal{ or } s= j \pm 1\\
1, &\textnormal{if } s= j \pm 2 \\
0, &\textnormal{otherwise}.
\end{cases} \nonumber\\
(X_2)_{js}&=
\begin{cases}
0, &\textnormal{if } (j,s)\in\{(1,2) ,(k,k-1)\} \textnormal{ or } s=j\\
1, &\textnormal{if } (j,s)\in\{(1,3),(k,k-2)\} \textnormal{ or } 3\leq s=j+1 \textnormal{ or } 3\leq s=j-1\leq k-2\\
2, &\textnormal{if } (j=1 \textnormal{ and } 4\leq s\leq k) \textnormal{ or } (j=k \textnormal{ and } 1\leq s\leq k-3) \\& \textnormal{or } 4 \leq s = j + 2 \textnormal{ or } s = j - 2\leq k-3\\
3, &\textnormal{otherwise}.
\end{cases} \nonumber\\
(X_3)_{js}&=
\begin{cases}
k-2, &\textnormal{if } (j,s)\in\{(1,1),(k,k)\}\\
k-3, &\textnormal{if } (j,s)\in\{(1,2),(2,1),(k-1,k), (k,k-1)\} \textnormal{ or } 2\leq j=s\leq k-1\\
k-4, &\textnormal{if } (j,s)\in\{(1,3),(3,1),(k-2,k), (k,k-2), (1,k), (k, 1)\} \\&\textnormal{or } 3\leq s=j+1\leq k-1 \textnormal{ or } 2\leq s=j-1\leq k-2\\
k-5, &\textnormal{if } (j=1 \textnormal{ and } 4\leq s\leq k-1) \textnormal{ or } (j=k \textnormal{ and } 2\leq s\leq k-3) \\
& \textnormal{or } (s=1 \textnormal{ and } 4\leq j\leq k-1)\textnormal{ or } (s=k \textnormal{ and } 2\leq j\leq k-3) \\
&\textnormal{or } 4\leq s=j+2\leq k-1 \textnormal{ or } 2\leq s=j-2\leq k-3\\
k-6, &\textnormal{otherwise}.
\end{cases} \nonumber
\end{align}
\begin{align}
&=1+ktx+x^{2}(t^{2}-2t)\Bigg(3\sum_{i=1}^{k}(A(x,t)^{-1})_{ii}+4\sum_{i=1}^{k-1}(A(x,t)^{-1})_{i,i+1}\nonumber\\&\hspace{217pt}+2\sum_{i=1}^{k-2}(A(x,t)^{-1})_{i,i+2}-2(A(x,t)^{-1})_{11}\Bigg)\nonumber\\
&+x^{2}\Bigg((6t+k-3)\mathbf{1}^{T}A(x,t)^{-1}\mathbf{1}-(4t-4)\sum_{i=1}^{k}(A(x,t)^{-1})_{1i}\nonumber\\&\hspace{180pt}-2\sum_{z=0}^{2}\sum_{i=1}^{k-1}\sum_{j=i+1+z}^{k}(A(x,t)^{-1})_{ij}-2(A(x,t)^{-1})_{11}\Bigg).\nonumber
\end{align}
At this point we use \eqref{e3}, \eqref{e4}, and \eqref{e5} and subsequently the identities for $q_0(k,x),\ldots, q_{41}(k,x)$. Then, after tedious algebraic manipulations, we arrive at the asserted formula.
\end{proof}

\begin{example}\label{e1}
The following three matrices correspond to the three multisets defined in the proof of Theorem \ref{t2}, for $k=8$.
\begin{align}
X_1&=
\begin{pmatrix}
2& 2& 1& 0& 0& 0& 0& 0\\
2& 3& 2& 1& 0& 0& 0& 0\\
1& 2& 3& 2& 1& 0& 0& 0\\
0& 1& 2& 3& 2& 1& 0& 0\\
0& 0& 1& 2& 3& 2& 1& 0\\
0& 0& 0& 1& 2& 3& 2& 1\\
0& 0& 0& 0& 1& 2& 3& 2\\
0& 0& 0& 0& 0& 1& 2& 2   
\end{pmatrix} \nonumber\\
X_2&=\begin{pmatrix}
0& 0& 1& 2& 2& 2& 2& 2\\
 1& 0& 1& 2& 3& 3& 3& 3\\
 2& 1& 0& 1& 2& 3& 3& 3\\
 3& 2& 1& 0& 1& 2& 3& 3\\
 3& 3& 2& 1& 0& 1& 2& 3\\
 3& 3& 3& 2& 1& 0& 1& 2\\
 3& 3& 3& 3& 2& 1& 0& 1\\
 2& 2& 2& 2& 2& 1& 0& 0  
\end{pmatrix}\nonumber\\
X_3&=\begin{pmatrix}
6& 5& 4& 3& 3& 3& 3& 4\\
 5& 5& 4& 3& 2& 2& 2& 3\\
 4& 4& 5& 4& 3& 2& 2& 3\\
 3& 3& 4& 5& 4& 3& 2& 3\\
 3& 2& 3& 4& 5& 4& 3& 3\\
 3& 2& 2& 3& 4& 5& 4& 4\\
 3& 2& 2& 2& 3& 4& 5& 5\\
 4& 3& 3& 3& 3& 4& 5& 6
\end{pmatrix}\nonumber
\end{align}
\end{example}

As a special case of Theorem \ref{t2} we obtain \cite[Theorem 3.2]{K}.
\begin{corollary}
Let $E(x)$ be the generating function of the number of cyclic staircase words over $k$. Then \[E(x) = 1+\frac{kx(1+3x)}{(1+x)(1-3x)}-\frac{2(k+1)xU_{k-1}\left(\frac{1-x}{2x}\right)}{(1+x)(1-3x)U_{k}\left(\frac{1-x}{2x}\right)}.\]    
\end{corollary}

\begin{proof}
Let $g_{m,n}$ denote the number of words $w\in[k]^n$ such that $\cs(w)=m$. Thus, $E(x)= \sum_{n\geq 0}g_{n,n}x^n$. Now, we have \[G(x,t)=\sum_{n\geq 0}\sum_{m=0}^{n}g_{m,n}t^mx^n.\] Thus, \[G\left(tx,\frac{1}{t}\right)=\sum_{n\geq 0}\sum_{m=0}^ng_{n,m}t^{n-m}x^n.\] Therefore, 
\begin{align}
E(x)&=G\left(tx,\frac{1}{t}\right)\Bigg|_{t=0}\nonumber\\&=1+\frac{kx(1+3x)}{(1+x)(1-3x)}-\frac{2(k+1)xU_{k-1}\left(\frac{1-x}{2x}\right)}{(1+x)(1-3x)U_{k}\left(\frac{1-x}{2x}\right)}.\nonumber\qedhere
\end{align}
\end{proof}

\begin{corollary}
The generating function of the number of cyclic Hertzsprung words over $k$ is given by 
\begin{align}
1&+\frac{1-3x}{(1+3x)(1-x)}\vast[\frac{1}{1-x\gamma(x,0)}\nonumber\\&+\frac{2x(k+1)}{(1-3x)U_{k}\left(-\frac{1+x}{2x}\right)}\left(\frac{1-x\frac{1+k-U_{k}\left(-\frac{1+x}{2x}\right)}{1+3x}}{1-x\gamma(x,0)}+U_{k-1}\left(-\frac{1+x}{2x}\right)-1\right)-kx-1\vast].\nonumber    
\end{align}
\end{corollary}

\begin{proof}
The assertion follows immediately from evaluating $G(x,0)$. 
\end{proof}

\begin{table}[H]
\begin{center}
\begin{tabular}{|c|c|}\hline
$k$ & Generating function \\\hline
$2$ & $1$\\[4pt]
$3$ & $-\frac{x^{2}+1}{(x-1)(x+1)}$\\[4pt]
$4$ & $\frac{-3x^{4}+3x^{2}+1}{(x^{2}-x-1)(x^{2}+x-1)}$\\[4pt]
$5$ & $\frac{-12x^{4}+4x^{3}+6x^{2}+1}{4x^{4}-2x^{3}-6x^{2}+1}$\\[4pt]
\hline
\end{tabular}
\caption{The generating function of the number of cyclic Hertzsprung words over $k$, for $k=2,3,4,$ and $5$.}\label{tab2}
\end{center}
\end{table}

\end{document}